\documentclass[11pt]{amsart}
\usepackage{lmodern}
\usepackage{amsmath, amsthm, amssymb, amsfonts}
\usepackage[normalem]{ulem}
\usepackage{hyperref}

\usepackage{mathrsfs}

\usepackage{verbatim} 
\usepackage{longtable}

\usepackage{mathtools}

\usepackage{tikz}
\usetikzlibrary{decorations.pathmorphing}
\tikzset{snake it/.style={decorate, decoration=snake}}

\usepackage{caption}

\usepackage{tikz-cd}
\usetikzlibrary{arrows}

\theoremstyle{plain}
\newtheorem{thm}{Theorem}[section]
\newtheorem{cor}[thm]{Corollary}
\newtheorem{lem}[thm]{Lemma}
\newtheorem{prop}[thm]{Proposition}
\newtheorem{conj}[thm]{Conjecture}
\newtheorem{question}[thm]{Question}

\theoremstyle{definition}

\theoremstyle{remark}
\newtheorem{rmk}[thm]{Remark}

\newcommand{\BC}{{\mathbb{C}}}

\newcommand{\BF}{{\mathbb{F}}}

\newcommand{\BP}{{\mathbb{P}}}
\newcommand{\BQ}{{\mathbb{Q}}}

\newcommand{\BZ}{{\mathbb{Z}}}

\newcommand{\CC}{{\mathcal C}}

\newcommand{\CE}{{\mathcal E}}
\newcommand{\CF}{{\mathcal F}}

\newcommand{\CH}{{\mathcal H}}
\newcommand{\CI}{{\mathcal I}}

\newcommand{\CL}{{\mathcal L}}
\newcommand{\CM}{{\mathcal M}}
\newcommand{\CN}{{\mathcal N}}
\newcommand{\CO}{{\mathcal O}}
\newcommand{\CP}{{\mathcal P}}

\DeclareFontFamily{OT1}{rsfs}{}
\DeclareFontShape{OT1}{rsfs}{n}{it}{<-> rsfs10}{}
\DeclareMathAlphabet{\curly}{OT1}{rsfs}{n}{it}

\usepackage{tikz}
\usepackage{lmodern}
\usetikzlibrary{decorations.pathmorphing}

\begin{document}
\title[Cohomology ring of the moduli of one-dimensional sheaves]{Generators for the cohomology ring of the moduli of one-dimensional sheaves on $\BP^2$}
\date{\today}

\newcommand\blfootnote[1]{%
  \begingroup
  \renewcommand\thefootnote{}\footnote{#1}%
  \addtocounter{footnote}{-1}%
  \endgroup
}

\author[W. Pi]{Weite Pi}
\address{Yale University}
\email{weite.pi@yale.edu}

\author[J. Shen]{Junliang Shen}
\address{Yale University}
\email{junliang.shen@yale.edu}


\begin{abstract}
We explore the structure of the cohomology ring of the moduli space of stable 1-dimensional sheaves on $\BP^2$ of any degree. We obtain a minimal set of tautological generators, which implies an optimal generation result for both the cohomology and the Chow ring of the moduli space. Our approach is through a geometric study of tautological relations.

\end{abstract}

\maketitle

\setcounter{tocdepth}{1} 

\tableofcontents
\setcounter{section}{-1}

\section{Introduction}

\subsection{Overview and motivation}
\label{intro}
Throughout, we work over the complex numbers $\BC$, and fix two integers $d$ and $\chi$ satisfying $d \geq 1$ and $\mathrm{gcd}(d,\chi)=1$.

The moduli space $M_{d,\chi}$ of stable 1-dimensional sheaves $\CF$ on $\BP^2$ with
\[
[\mathrm{supp}(\CF)] = dH \in  H_2(\BP^2, \BZ), \quad \chi(\CF) = \chi
\]
is a nonsingular irreducible projective variety of dimension $d^2+1$ \cite{LeP}. Here $H$ is the class of a line, $\mathrm{supp}(-)$ denotes the Fitting support, and the stability is with respect to the slope
\[
\mu( \CE) = \frac{\chi(\CE)}{c_1(\CE)\cdot H} \in \BQ.
\]

Geometry and topology of the moduli spaces $M_{d,\chi}$ have been studied intensively for decades from the perspectives of strange duality \cite{YY1,YY2,YY3}, birational geometry \cite{Woolf},  enumerative geometry of local $\BP^2$ \cite{CC15, PB, PB2, BFGW, MS_GT, YY4}, and so on. Much effort has been made to explicitly describe $M_{d,\chi}$ for low values of $d$; see for example \cite{DM11, Mai11, Mai13, CM14, ChM, BMW, YY0}. 

The moduli spaces $M_{d,\chi}$ share similar features with another type of interesting moduli spaces --- moduli of Higgs bundles. By the BNR correspondence \cite{BNR}, the moduli of stable Higgs bundles on a Riemann surface $C$ can be viewed as the moduli of stable 1-dimensional sheaves on the surface $T^*C$, and the Hitchin system \cite{Hit, Hit1} is exactly the Hilbert--Chow morphism sending a 1-dimensional sheaf to its Fitting support; this is analogous to the map
\begin{equation}\label{Hilb-Chow}
h: M_{d,\chi} \to \BP H^0(\BP^2, \CO_{\BP^2}(d)), \quad \CF \mapsto \mathrm{supp}(\CF)
\end{equation}
associated with the moduli space $M_{d,\chi}$. The recent work \cite{MS_GT} provides further evidence to this analogy where the decomposition theorem \cite{BBD} and the support theorem \cite{Ngo, CL, MS_GT} were applied to study (intersection) cohomology groups for both types of moduli spaces.

On the other hand, the ring structure of the cohomology of the Higgs moduli space has formed a rich subject to study. The cohomology of the Higgs moduli space is known to be generated by \emph{tautological classes} \cite{Markman}. A complete set of relations was found in the case of rank two \cite{HT2}, which relies on earlier work of the tautological relations for the moduli of vector bundles \cite{Kir}. All these results play an important role in some recent progress of the $P=W$ conjecture \cite{dCHM1} which is a deep connection between the cohomology of the Higgs moduli space and the Hodge theory of character variety; see also \cite{dCMS, dCSZ}.

Our main interest here is to explore the ring structure for the cohomology of $M_{d,\chi}$ in terms of the tautological classes. A better understanding for $M_{d,\chi}$ may shed new light on understanding the cohomology ring of the moduli of stable 1-dimensional sheaves on a surface, which includes the moduli space of Higgs bundles.

Before introducing more notation, we first state a brief version of our main result. As explained in Remark \ref{rmk1.6} $H^*(M_{d,\chi})$ has no odd class. Furthermore, by \cite[Theorem 2]{Integral} the cycle class map from Chow to cohomology is an isomorphism. Hence we use $A^*(-)$ to denote the even cohomology $H^{2*}(-, \BQ)$, or equivalently, the Chow ring $\mathrm{CH}^*(-)$ with $\BQ$-coefficients. We denote by $R^k(-)$ the subalgebra of $A^*(-)$ generated by the classes in $A^{\leq k}(-)$.

\begin{thm}\label{thm0}
For $d\geq 3$, we have\footnote{When $d=1,2$, the moduli space $M_{d,\chi}$ is a projective space $\BP H^0(\BP^2,\CO_{\BP^2}(d))$ whose cohomological structure is clear.}
\[
R^{d-3}(M_{d,\chi}) \subsetneqq R^{d-2}(M_{d,\chi}) = A^*(M_{d,\chi}).
\]
\end{thm}

The theorem asserts that the entire cohomology of the $(d^2+1)$-dimensional variety $M_{d,\chi}$ is generated by classes of (algebraic) degrees $\leq d-2$, and this bound is optimal. A more detailed version of Theorem \ref{thm0} is given by Theorem \ref{thm1} where we describe a minimal set of tautological generators for $A^*(M_{d,\chi})$.

\subsection{Tautological classes}
Let $\BF$ be a universal family over $\BP^2 \times M_{d,\chi}$; it is a torsion sheaf supported on a divisor. For a stable sheaf $[\CF] \in M_{d,\chi}$, the restriction of $\BF$ to the fiber $\BP^2 \times [\CF]$ recovers $\CF$.

A natural way to construct classes in $A^*(M_{d,\chi})$ is to integrate $\mathrm{ch}_{1+k}(\BF)$ over a class $H^j \in A^j(\BP^2)$. While the choice of a universal family is not unique, we may normalize the universal class $\mathrm{ch}_{k+1}(\BF) \in A^k(\BP^2 \times M_{d,\chi})$ and integrate it over $H^j$; this gives rise to the tautological classes 
\[
c_k(j) \in A^{k+j-1}(M_{d,\chi}), \quad j \in \{0,1,2\};
\]
see Section \ref{Sec1.1}. Here the normalization is characterized by $c_1(0) = c_1(1) = 0$. These classes behave nicely under natural symmetries of the moduli spaces $M_{d,\chi}$ (Proposition \ref{prop1.4}).

\medskip

Our main result is the following.

\begin{thm}\label{thm1} Assume $d \geq 3$.We have:
\begin{enumerate}
    \item[(a)] $A^*(M_{d,\chi})$ is generated over $\BQ$ by the $3d-7$ classes of degrees $\leq d-2$:
    \begin{equation}\label{taut}
   c_0(2), c_2(0) \in A^1(M_{d,\chi}), \quad  c_{k}(0), c_{k-1}(1), c_{k-2}(2) \in A^{k-1}(M_{d,\chi}),~~~ k\in \{3,\dots, d-1 \}.
    \end{equation}
    \item[(b)] There is no relation among these $3d-7$ classes of (\ref{taut}) in degrees $\leq d-2$.
\end{enumerate}

\end{thm}

Theorem \ref{thm1} implies that any set of generators of $A^*(M_{d,\chi})$ contains at least $3d-7$ elements, and (\ref{taut}) provides a minimal one. It is clear that Theorem \ref{thm0} follows immediately from Theorem \ref{thm1}. As we discuss further in Section \ref{sec3.2}, numerical data for low values of $d$ suggest that there is no relation also in degree $d-1$ and nontrivial relations start in degree $d$.\footnote{See Remark \ref{llast} for an update on this.}

The proof of Theorem \ref{thm1} consists of the following main ingredients: (i) By a result of Beauville \cite{Beau} the ring $A^*(M_{d,\chi})$ is generated by tautological classes. (ii) Using the geometry of stable 1-dimensional sheaves we produce tautological relations; this allows us to express any tautological class in terms of the first $3d-7$ ones. (iii) Finally, we obtain the freeness of these $3d-7$ classes by a Betti number constriant where we apply a recent result of Yuan \cite{YY4}.

\subsection{Enhanced cohomological $\chi$-independence}\label{sec0.3}
The moduli spaces $M_{d,\chi}$ admits a mysterious symmetry predicted by the consideration from enumerative geometry \cite{PB, Toda, MS_GT}, called the cohomological $\chi$-independence. More precisely, it was proven in \cite{MS_GT} that, for any two integers $\chi, \chi'$ (not necessarily coprime to $d$), we have
\[
\mathrm{IH}^*(M_{d,\chi}) \simeq \mathrm{IH}^*(M_{d,\chi'})
\]
preserving the perverse and the Hodge filtrations. Here $\mathrm{IH}^*(-)$ stands for intersection cohomology. If we restrict ourselves to the case $(\chi,d) = (\chi',d) = 1$,  intersection cohomology coincide with singular cohomology which admits a canonical $\BQ$-algebra structure. It is natural to ask if in this case the cohomological $\chi$-independence can be strengthened to an isomorphism of $\BQ$-algebras.

\begin{question}\label{Q}
For any $\chi, \chi'$ coprime to $d$, is there an isomorphism 
\begin{equation*}\label{Q'}
H^*(M_{d,\chi}) \simeq H^*(M_{d,\chi'})
\end{equation*}
of $\BQ$-algebras?
\end{question}

Since we have obtained a minimal set of generators whose number is independent of $\chi$, it suffices to understand the dependence on $\chi$ of the tautological relations among (\ref{taut}).

The parallel statements of Question \ref{Q} and its enhancements hold for moduli of Higgs bundles as predicted by the $P=W$ conjecture; their proofs \cite{dCZ, dCSZ} rely on techniques in characteristic $p$.

\subsection{Acknowledgements}
 We would like to thank Pierrick Bousseau, Jinwon Choi, and Kiryong Chung for very helpful discussions, and in particular to Jinwon Choi for help on some numerical data relevant to the discussion in Section \ref{sec3.2}. We also thank the anonymous referee for careful reading of the paper. W.P. would like to thank Yuxiao Feng for a helpful code. J.S. was supported by the NSF grant DMS-2134315.

\section{Tautological classes and normalizations}
Throughout, we assume $d \geq 3$. We introduce (normalized) tautological classes for $M_{d,\chi}$. The construction of the normalization is parallel to \cite[Section 0.3]{dCMS}.

\subsection{Normalizations}\label{Sec1.1}
 Under the assumption $\mathrm{gcd}(d,\chi)=1$, there is a universal family (see \cite[Theorem 4.6.5]{HL}) over $\BP^2 \times M_{d,\chi}$, which we denote by $\BF$. Let
\[
\pi_P: \BP^2 \times M_{d,\chi} \to \BP^2, \quad \pi_M:  \BP^2 \times M_{d,\chi} \to M_{d,\chi}
\]
be the projections. The choice of $\BF$ is not unique; for another universal family $\BF'$ there exists a line bundle $L \in \mathrm{Pic}(M_{d,\chi})$ with 
\[
\BF = \BF' \otimes \pi_M^*L.
\]

Nevertheless, as a universal sheaf $\BF$ is a torsion sheaf on $\BP^2 \times M_{d,\chi}$ supported on a divisor, its first Chern character, which records its support, is uniquely determined.

\begin{lem}\label{lem1.1}
Let $\BF$ be a universal family. We have 
\[
\mathrm{ch}_1(\BF) = \widetilde{h}^*c_1(\CO(d,1)) \in A^1(\BP^2 \times M_{d,\chi}).
\]
Here $\widetilde{h}: \BP^2 \times M_{d,\chi} \to \BP^2 \times \BP H^0(\BP^2, \CO_{\BP^2}(d))$ is induced by (\ref{Hilb-Chow}). In particular $\mathrm{ch}_1(\BF) $ does not depend on the choice of $\BF$.
\end{lem}

\begin{proof}
A universal sheaf $\BF$ is a torsion sheaf on $\BP^2 \times M_{d,\chi}$ supported on a divisor. Hence its first Chern character recovers its support, which is the pullback via $\widetilde{h}$ of the incidence variety whose class is given by
\[
c_1(\CO(d,1)) \in A^1\left( \BP^2 \times \BP H^0(\BP^2, \CO_{\BP^2}(d))\right). \qedhere
\]
\end{proof}

The class $\mathrm{ch}_2(\BF)$ is dependent on the choice of $\BF$, which we use to conduct the normalization. For a universal family $\BF$ and a class 
\[
\alpha= \pi_P^* \alpha_P + \pi_M^* \alpha_M \in A^1(\BP^2\times M_{d,\chi}), \quad \textup{with}~~ \alpha_P \in A^1(\BP^2), ~~~\alpha_M\in A^1(M_{d,\chi}),
\]
we consider the twisted Chern character
\[
\mathrm{ch}^\alpha(\BF) := \mathrm{ch}(\BF) \cdot \mathrm{exp}(\alpha),
\]
and we denote $\mathrm{ch}_k^\alpha(\BF)$ its degree $k$-part. For any $\gamma \in A^*(\BP^2)$, we set
\[
\int_\gamma \mathrm{ch}_k^\alpha(\BF): =\pi_{M*}\left( \pi_P^* \gamma \cdot \mathrm{ch}_k^\alpha(\BF)\right) \in A^*(M_{d,\chi}).
\]

\begin{prop}\label{prop1.2}
Let $\BF$ be a universal family. There exists a unique class $\alpha$ as above such that
\[
\int_{H} \mathrm{ch}_2^\alpha(\BF) = 0, \quad \int_{\mathbf{1}_{\BP^2}} \mathrm{ch}_2^\alpha(\BF) = 0.
\]
\end{prop}

\begin{proof}
We have
\[
\mathrm{ch}_2^\alpha(\BF) = \mathrm{ch}_2(\BF) + \alpha \cdot \mathrm{ch}_1(\BF).
\]
Hence the conditions read
\begin{equation}\label{1.2}
    \int_H \mathrm{ch}_2(\BF) = -\int_H \alpha \cdot \mathrm{ch}_{1}(\BF), \quad \int_{\mathbf{1}_{\mathbb{P}^2}} \mathrm{ch}_2(\BF) = - \int_{\mathbf{1}_{\mathbb{P}^2}} \alpha \cdot \mathrm{ch}_1(\BF).
\end{equation}
By \cite{LeP, Woolf}, the Picard group of $M_{d,\chi}$ is spanned by 2 classes, which we denote by $D_0$ and $D_1$. Without loss of generality we set $D_0$ to be the pullback of the hyperplane class via (\ref{Hilb-Chow}) and $D_1$ to be a relative ample class. Therefore Lemma \ref{lem1.1} implies that 
\[
\mathrm{ch}_1(\BF) = d \cdot \pi_P^*H + \pi_M^* D_0.
\]
We may assume that 
\[
\alpha = \lambda_1 \cdot \pi_P^*H + \lambda_2 \cdot\pi_M^* D_0 + \lambda_3 \cdot \pi_M^* D_1.
\]
Then by a direct calculation, $\lambda_1$ is  determined by the second equation of (\ref{1.2}) and then the first equation determines $\lambda_2, \lambda_3$.
\end{proof}

We define the tautological class
\[
c_k(j) : = \int_{H^j} \mathrm{ch}^\alpha_{k+1} (\BF) \in A^{k+j-1}(M_{d,\chi})
\]
where $\alpha$ is the class characterized by Proposition \ref{prop1.2}.

\subsection{Properties}
We first summarize some properties of the tautological classes. The following is an immediate consequence of Lemma \ref{lem1.1} and Proposition \ref{prop1.2}.

\begin{prop}\label{prop1.3}
Let $c_k(j)$ be the tautological classes defined above from a universal family $\BF$ over $\BP^2 \times M_{d,\chi}$. Then:
\begin{enumerate}
    \item[(a)] $c_k(j)$ does not depend on the choice of $\BF$.
    \item[(b)] We have 
    \[
    c_1(0) = 0 \in A^0(M_{d,\chi}), \quad c_1(1) = 0 \in A^1(M_{d,\chi}), \quad c_0(1)= d \in A^0(M_{d,\chi}).
    \]
    \item[(c)] The Picard group of $M_{d,\chi}$ is spanned by $c_0(2)$ and $c_2(0)$; the class $c_0(2)$ recovers $D_0$ in the proof of Proposition \ref{prop1.2} and $c_2(0)$ is a relative ample class with respect to (\ref{Hilb-Chow}).
\end{enumerate}
\end{prop}

\begin{proof}
The first two claims and $c_0(2) = D_0$ follow from the definition together with Lemma \ref{lem1.1} and Proposition \ref{prop1.2}. We now prove that $c_2(0)$ is a relative ample class. Since the Picard group of $M_{d,\chi}$ is spanned by 2 classes, it suffices to verify that the restriction of $c_2(0)$ to a smooth fiber of $h$ is non-trivial in cohomology.

We take $C \subset \BP^2$ a nonsingular degree $d$ curve which represents a point $[C]$ on the target of $h$. The fiber $h^{-1}([C])$ is isomorphic to the Jacobian $\mathrm{Jac}(C)$. Moreover, the restriction of a universal family to $\BP^2 \times \mathrm{Jac}(C)$ is recovered by $i_* \left(\CP \otimes \pi_C^*L_C\right)$ where $i$ is the embedding $C\times \mathrm{Jac}(C) \hookrightarrow \BP^2 \times \mathrm{Jac}(C)$, $\CP$ is the normalized Poincar\'e line bundle on $C \times \mathrm{Jac}(C)$, and $L_C$ is a line bundle on $C$. The desired property follows from the fact that $c_1(\CP)^2$ has nontrivial K\"unneth component in $H^2(C)\otimes H^2(\mathrm{Jac}_C)$.
\end{proof}

Secondly, we note that the tautological classes behave nicely under the following two types of symmetry carried by the moduli spaces $M_{d,\chi}$:
\begin{enumerate}
    \item[(i)] The first type of symmetry is given by the isomorphism
    \[
    \phi_1: M_{d,\chi} \xrightarrow{\simeq} M_{d,\chi+d}, \quad \CF \mapsto \CF\otimes \CO_{\BP^2}(1).
    \]
    \item[(ii)] The second type of symmetry is given by the isomorphism \cite{Mai10}
    \[
    \phi_2: M_{d,\chi} \xrightarrow{\simeq} M_{d, -\chi}, \quad \CF \mapsto \CE \kern -1.5 pt \mathit{xt}^1(\CF, \omega_{\BP^2}).
    \]
\end{enumerate}

\begin{prop}\label{prop1.4}
We have
\[
\phi_1^* c_k(j) = c_k(j), \quad \phi_2^* c_k(j) = (-1)^kc_k(j).
\]
\end{prop}

\begin{proof}
The first identity follows from the fact that the pullback of a universal family over $M_{d,\chi+d}$ along
\[
\mathrm{id} \times \phi_1: \BP^2\times M_{d,\chi} \rightarrow \BP^2 \times M_{d,\chi+d}
\]
is of the form $\BF \otimes \pi_P^* \CO_{\BP^2}(1)$ with $\BF$ a universal family for the moduli space $M_{d,\chi}$. To see the second identity, we note that the pullback of a universal family over $M_{d,-\chi}$ along $\mathrm{id}\times \phi_2$ is of the form $R\CH \kern -1.2 pt\mathit{om}(\BF, \pi_P^*\omega_{\BP^2})[1]$ in the bounded derived category; see \cite{Mai10}. Its Chern character is
\[
\mathrm{ch}\left(-R\CH \kern -1.2 pt \mathit{om}(\BF, \pi_ P^*\omega_{\BP^2})\right) = \left( \sum_{i\geq 0} (-1)^k\mathrm{ch}_{1+k}(\BF) \right) \cdot \mathrm{exp}\left( \pi_P^*c_1(\omega_{\BP^2})\right).
\]
The claim then follows from definition of the normalization.
\end{proof}

We conclude this section by recalling the following theorem by  Beauville \cite{Beau}; see also Ellingsrud--Str\o{}mme \cite{ES}.

\begin{thm}[\cite{Beau}]\label{thm1.5}
The ring $A^*(M_{d,\chi})$ is generated over $\BQ$ by the tautological classes $c_k(j)$.
\end{thm}

\begin{proof}
Take a universal family $\BF$ over $\BP^2 \times M_{d,\chi}$. It was proven in \cite{Beau} that $A^*(M_{d,\chi})$ is generated by the classes
\[
\int_{H^j} \mathrm{ch}_{k+1}(\BF) = \pi_{M*}\left( \pi_P^* H^j \cdot \mathrm{ch}_{k+1}(\BF) \right) \in A^*(M_{d,\chi}).
\]
By a direct calculation (see also Proposition \ref{prop2.2}), these classes can be expressed in terms of the normalized tautological classes $c_k(j)$ and the classes in $A^1(M_{d,\chi})$; the latter are also tautological by Proposition \ref{prop1.3} (c).
\end{proof}

\begin{rmk}\label{rmk1.6}
As an immediate consequence of Theorem \ref{thm1.5}, we obtain that $H^{2k+1}(M_{d,\chi}) = 0$ for any $k$; moreover $H^{i,j}(M_{d,\chi}) =0$ if $i \neq j$. Using the $\chi$-independence result \cite{MS_GT}, this further implies that \[
\mathrm{IH}^{i,j}(M_{d,\chi}) = 0, \quad \textup{if}~~ i\neq j
\]
for any $\chi$ not necessarily coprime to $d$. We refer to \cite[Theorem 0.4.1]{PB} for another proof of this result.
\end{rmk}

In view of Theorem \ref{thm1.5}, we explain in Section \ref{Sec2} a method to extract a minimal set of generators from all tautological classes.

\subsection{An example: $d=3$}\label{deg3} We work out the case $d=3$ in detail and prove Theorem \ref{thm1} in this case. By the symmetry of Proposition \ref{prop1.4}, we only need to treat the case $\chi= -1$. We write $M: = M_{3,-1}$ for notational convenience.

\begin{prop}\label{main} We have
\[
A^*(M_{3,-1}) = \BQ[c_0(2), c_2(0)]/I
\]
where $I$ is a homogeneous ideal generated by two elements in degrees 3 and 9 respectively.
\end{prop}

In particular $A^*(M)$ is generated by the two classes in $A^1(M)$, and the lowest degree nontrivial relation occurs in degree $3$. This matches with the prediction of Theorem \ref{thm1} (and Conjecture \ref{conj3.3} below).

\smallskip

 Recall that by \cite{LeP} the moduli space $M_{3,2}$ (hence also $M$) is  isomorphic to the universal cubic
 \[
\mathcal{C} \subset \mathbb{P}^2\times \BP H^0(\BP^2, \CO_{\BP^2}(3))=\mathbb{P}^2\times {\BP}^9
\]
where we use $H_1$ and $H_2$ to denote the hyperplane classes in the first and the second factors respectively. Projecting over the second factor, we see that $M$ is the projective bundle $\BP(\CE)$ associated with a rank $9$ vector bundle $\CE$. Here the vector bundle $\CE$ is characterized by
\[
\CE|_{x} = H^0(\BP^2, \CI_x\otimes \CO_{\BP^2}(3))
\]
with $\CI_x \subset \CO_{\BP^2}$ the ideal sheaf of the point $x$. In particular, we have the short exact sequence
\[
0 \to \CE \to \CO_{\BP^2}^{\oplus 10} \to \CO_{\BP^2}(3) \to 0,
\]
which yields 
\[
c_1(\CE) = -3H_1, \quad c_2(\CE) = 9H_1^2.
\]
Denote by $\xi\in A^1(M)$ the relative hyperplane class over $\BP^2$. By the projective bundle formula, we obtain that \[
A^*(M) = A^*(\CC)=\BQ[H_1, \xi]/(H_1^3, \xi^9 - 3H_1\xi^8+9H_1^2\xi^7).
\]
This already implies Proposition \ref{main} in view of Proposition \ref{prop1.3} (c).

For completion, we give an explicit expression of the ideal $I$ in terms of the tautological generators $c_0(2), c_2(0)$. Consider the incidence subvarieties
\begin{align*}
    W =& \left\{(x,x,y) \in \BP^2 \times \BP^2 \times \BP^9:~ (x,y) \in \CC\right\} \subset \BP^2 \times M,\\
    V = & \left\{(x,y,z) \in \BP^2 \times \BP^2 \times \BP^9:~ (x,z) \in \CC,~ (y,z)\in \CC \right\} \subset \BP^2 \times M.
\end{align*}
Clearly $W \subset V$. Recall that $H$ is the class of a line on the first factor $\BP^2$. The Chern character of the structure sheaf $\CO_V$ can be obtained by applying the Grothendieck-Riemann-Roch theorem to the closed embedding $i_V: V\hookrightarrow \BP^2\times M$:
\[
\mathrm{ch}(\CO_V) = i_{V*}\mathrm{td}({T_{i_V}})= (3H+H_2)-\frac{1}{2}(3H+H_2)^2 +\frac{1}{6}(3H+H_2)^3- \cdots.
\]
Similarly we also have for $i_W: W\hookrightarrow \BP^2 \times M$:
\[
\mathrm{ch}(\CO_W) = (H+HH_1+H_1^2) -\frac{3}{2}(H^2H_1+HH_1^2) + \cdots.
\]
The ideal sheaf associated with $W \subset V$ yields a universal family over $\BP^2 \times M$:
\[
0 \to \BF \to \CO_V \to \CO_W \to 0
\]
from which we may calculate the Chern characters : 
\begin{align*}
    \mathrm{ch}_1(\BF)&=3H+H_2 ,\\
    \mathrm{ch}_2(\BF)&= -\frac{11}{2}H^2 - (H_1+3H_2)H - (H_1^2+\frac{1}{2}H_2^2),\\
    \mathrm{ch}_3(\BF)&=  \frac{1}{6}(3H+H_2)^3 +\frac{3}{2}(H^2H_1+HH_1^2), \cdots.
\end{align*} 
We obtain that the class of the normalization is
\[
\alpha = \frac{11}{6}H + \frac{1}{3}H_1 + \frac{7}{18}H_2, 
\]
and the tautological classes are
\begin{equation}\label{last}
c_0(2) = H_2,\quad c_2(0) = - \frac{1}{3}H_1 +\frac{49}{72}H_2,  \quad c_1(2) = -H_1^2 +\frac{1}{3}H_1 H_2 - \frac{1}{9} H_2^2,\, \cdots .
\end{equation}
Finally we note that $\xi = H_2$. Therefore the ideal
\[
I = (H_1^3, \xi^9-3H_1 \xi^8 +9H_1^2 \xi^7)
\]
can be expressed in terms of $c_0(2),c_2(0)$ via (\ref{last}).

\begin{rmk}
For the case $d=4$, the Chow ring of $M_{4,1}$ was calculated by Chung-Moon \cite{ChM}. In particular, they showed that $A^*(M_{4,1})$ is generated as a $\BQ$-algebra by 2 generators in $A^1(M_{4,1})$ and 3 generators in $A^2(M_{4,1})$, which matches with Theorem \ref{thm1}.  
\end{rmk}

\section{Tautological relations}\label{Sec2}

\subsection{Overview}
The goal of this section is to prove Theorem \ref{thm1} (a). As the structure for $A^*(M_{3,\chi})$ is clear by Section \ref{deg3}, from now on we focus on the case $d \geq 4$. 

In view of Theorem \ref{thm1.5}, we show that every class 
\[
c_k(j)\in A^{k+j-1}(M_{d,\chi}), \quad k+j-1 \geq d-1
\]
can be expressed in terms of (\ref{taut}).

For convenience, we fix a universal family $\BF$, and define the classes
\[
e_k(j):=(-1)^{k+1}\int_{H^j} \mathrm{ch}_{k+1}(\BF) \in A^{k+j-1}(M_{d,\chi})
\]
which are dependent on $\BF$.\footnote{The factor $(-1)^{k+1}$ comes from taking dual, as we shall see below.} We introduce a total ordering $\prec$ on the double indices $(k,j)$: we say that $(k,j) \prec (k',j')$ if and only if $k+j-1<k'+j'-1$, or $k+j-1=k'+j'-1$ and $k<k'$. Thus we can talk about the \textit{leading term} of a homogeneous polynomial in the classes $e_k(j)$, respectively, $c_k(j)$, in $A^*(M_{d,\chi})$.

\medskip

Consider the $3d-6$ classes
\begin{equation}\label{eclass}
   e_{k}(0), e_{k-1}(1), e_{k-2}(2) \in A^{k-1}(M_{d,\chi}),~~~ k\in \{2,3,\dots, d-1 \}.
    \end{equation}

\begin{lem}
\label{lem2.1}
    We have 
    \[
    e_0(1)= -d\in A^0(M_{d,\chi}), \quad e_1(0)=\chi-\frac{3}{2}d\in A^0(M_{d,\chi}).
    \]
\end{lem}

\begin{proof}
    The first claim is established in Lemma \ref{lem1.1}. For the second, consider a sheaf $[\CF]\in M_{d,\chi}$; the Hirzebruch-Riemann-Roch formula gives
    \[
    \chi(\CF)=\deg_2(\mathrm{ch}(\CF)\cdot \mathrm{td}(\BP^2))=\deg\Big(\frac{3}{2}  \mathrm{ch}_1(\CF)\cdot H +\mathrm{ch}_2(\CF)\Big),
    \]

Hence $\mathrm{ch}_2(\CF)=(\chi-\frac{3}{2}d)H^2\in A^2(M_{d,\chi})$, and the claim follows.
\end{proof}

The next proposition follows from a direct calculation via the expansion 
\[\mathrm{ch}^\alpha_{k+1}(\mathbb{F})=\mathrm{ch}_{k+1}(\BF)+\alpha \cdot  \mathrm{ch}_k(\BF)+\frac{1}{2}\alpha^2 \mathrm{ch}_{k-1}(\BF)+\cdots
    \] 
as for the proof of Theorem \ref{thm1.5}.

\begin{prop}
\label{prop2.2}
    The $3d-7$ tautological classes in (\ref{taut}) and the $3d-6$ classes in (\ref{eclass}) generate each other. Moreover, every class $c_k(j)\in A^{k+j-1}(M_{d,\chi})$ with $k+j-1\geq 2$ can be expressed as a polynomial in (\ref{taut}) with leading term $e_k(j)$, and vice versa.
\end{prop}

For any $k,j$ with $k+j-1\geq d-1$, we will produce a relation in $A^*(M_{d,\chi})$ with leading term $e_k(j)$. We say that such a relation \textit{kills} $e_k(j)$. Theorem \ref{thm1} (a) follows from the existence of the relations that kill $e_j(k)$ as above.
    
\subsection{Producing relations}

By Proposition \ref{prop1.4}, we only need to consider the case $0<\chi<d$ so that any $\CF \in M_{d,\chi}$  satisfies $0<\mu(\CF)<1$.

We consider the triple product $Y:=\BP^2 \times M_{d,\chi} \times \check{\mathbb{P}}^2$ with $\check{\mathbb{P}}^2$ the dual projective plane. Let $\pi_R : Y\to \check{\mathbb{P}}^2$ be the projection to the third factor. We write $p=\pi_P\times \pi_M$, $q=\pi_P\times \pi_R$, and $r=\pi_M \times \pi_R$, where by abuse of notation we denote by $\pi_P$ and $\pi_M$ also the projections from $Y$ to $\BP^2$ and $M_{d,\chi}$ respectively:
 
 \[\begin{tikzcd}
	& Y \\
	{\BP^2\times M_{d,\chi}} & {M_{d,\chi} \times \check{\mathbb{P}}^2} & {\BP^2\times \check{\mathbb{P}}^2}
	\arrow["p"', from=1-2, to=2-1]
	\arrow["r"', from=1-2, to=2-2]
	\arrow["q", from=1-2, to=2-3].
\end{tikzcd}\] 
Let  $Z\subset \BP^2\times \check{\mathbb{P}}^2$ be the incidence subscheme, and let $\CO_Z$ be the structure sheaf of $Z$ viewed as a coherent sheaf on $\BP^2 \times \check{\BP}^2$. Let $\beta \in A^1(\check{\BP}^2)$ be the class of a line in $\check{\BP}^2$.

 \begin{lem}
 \label{lem2.4}
 We have
 \[
 \mathrm{ch}(q^*\mathcal{O}_Z)=1-\mathrm{exp}(-(\pi_P^* H + \pi_R^* \beta)).
 \]
 \end{lem}

\begin{proof}
It follows from the ideal sheaf sequence
 \[
 0 \to \CO_{\BP^2\times \check{\BP}^2}(-Z) \to \CO_{\BP^2 \times \check{\BP}^2 }\to \mathcal{O}_Z \to 0,
 \]
 and the fact that the divisor class $[Z]\in A^1(\BP^2\times \check{\BP^2})$ is $\pi_P^* H +\pi_R^* \beta$.
 \end{proof}

We consider the complex
 \[
 \mathcal{H}(n):= R \mathcal{H}\kern -1.2 pt \mathit{om}(p^*\BF, q^*\mathcal{O}_Z\otimes \pi_P^*\mathcal{O}_{\BP^2}(-n)) \in D^b\mathrm{Coh}(Y)
 \]
 for $n\in \{1,2,3\}$. Since $r: Y \to M_{d,\chi} \times \check{\BP}^2$ is a trivial $\BP^2$-bundle, the derived pushforward of $\CH(n)$:
 \[
 Rr_* \mathcal{H}(n) \in D^b\mathrm{Coh}(M_{d,\chi} \times \check{\BP}^2)
 \]
 admits a three-term resolution $K^0\to K^1\to K^2$ by vector bundles.
 
 \begin{lem}
 \label{K^1}
For each $n\in \{1,2,3\}$, we can choose $K^i$ such that $K^0=K^2=0$ and $K^1$ is free of rank $d$.
 \end{lem}

\begin{proof}
 For a point 
 \[
 P=([\mathcal{F}],p)\in M_{d,\chi}\times \check{\mathbb{P}}^2,
 \]
 we denote by $H_P \subset \BP^2$ the line correspoinding to $p \in \check{\BP}^2$. Then over this point $P$, cohomology of the complex of 
 \begin{equation}\label{66}
 K^0(P) \to K^1(P)\to K^2(P)
 \end{equation}
 computes the extension groups
 \[
 \mathrm{Ext}^i(\mathcal{F}, \mathcal{O}_{H_P}(-n)) \quad i=1,2,3,
 \]
 on $\BP^2$; see \cite{Beau}. Note $\mu(\mathcal{O}_{H_P}(-n))=1-n$ so that $\mu(\mathcal{O}_{H_P}(-n))<\mu(\CF)$ for $n\in \{1,2,3\}$. Therefore by stability we have $\mathrm{Hom}(\CF, \mathcal{O}_{H_P}(-n))=0$. On the other hand, Serre duality gives
 \[
 \mathrm{Ext}^2(\mathcal{F}, \CO_{H_P}(-n))\simeq \mathrm{Hom}(\CO_{H_P}(-n), \CF\otimes \omega_X)^\vee=\mathrm{Hom}(\CO_{H_P}(-n), \CF(-3))^\vee.
 \]
 Note that $\mu(\CF(-3))< -2\leq 1-n$ for $n\in \{1,2,3\}$, and hence $\mathrm{Hom}(\CO_{H_P}(-n), \CF(-3))=0$, again from stability. It follows that the zeroth and the second cohomology of (\ref{66}) vanish for every $P$. Hence $Rr_*\mathcal{H}(n)$ can be represented by a single vector bundle $K^1$ concentrated in degree $1$ whose rank is determined by the Hirzebruch-Riemann-Roch calculation:
 \[
 \chi(\CH\kern -1.2pt \mathit{om}(\CF, \CO_{H_P}(-n)))=\deg_2(\mathrm{ch}(\CF^\vee)\cdot \mathrm{ch}(\CO_{H_P}(-n)) \cdot \mathrm{td}(\BP^2))=-d. \qedhere
 \]
\end{proof}

The next corollary follows immediately.

\begin{cor}
\label{cor2.3}
For $\ell \geq d+1$ and $n\in \{1,2,3\}$, we have
\[
 c_\ell(-Rr_*\mathcal{H}(n))=0 \in A^{2\ell}(M_{d,\chi} \times \check{\BP}^2).
\]
\end{cor}

Corollary \ref{cor2.3} is the main source producing relations, which we explain as follows.

\medskip
 
By the Grothendieck-Riemann-Roch theorem, we have
\begin{align*}
\mathrm{ch}(Rr_*\mathcal{H}(n))&=r_*(\mathrm{ch}(\mathcal{H}(n))\cdot \mathrm{td}(\BP^2))\\
&=r_*(\mathrm{ch}(p^*\mathbb{F}^\vee)\cdot \mathrm{ch}(q^*\mathcal{O}_Z\otimes \pi_P^* \mathcal{O}_{\BP^2}(-n))\cdot \mathrm{td}(\BP^2)).
\end{align*}
If we expand the right-hand side, Corollary \ref{cor2.3} gives relations among the pullbacks of $e_k(j)$ and $\beta^i$ in $A^*(M_{d,\chi} \times \check{\BP}^2)$. This will be the main result we use in Section \ref{Sec2.3}, so we state it as a proposition below.

\begin{prop}
For every $\ell \geq d+1$, the following identity holds:
 \begin{equation}
 \label{relation}
 \sum_\mathbf{m} \prod_{s=1}^\ell \frac{((s-1)!)^{m_s}}{(m_s)!}\left(\pi_M^* A_s-\sum_{\substack{i+j=s\\ 0\leq i\leq 2}} \frac{\pi_R^*\beta^i}{i!}(-1)^i \pi_M^* B_j\right)^{m_s}=0 \in A^*(M_{d,\chi} \times \check{\mathbb{P}}^2).
 \end{equation}
 Here, the first sum is over all $\ell$-tuple of non-negative integers $\mathbf{m}=(m_1,m_2,\ldots, m_\ell)$ such that $m_1+2m_2+\cdots +\ell m_\ell=\ell$, and $A_s, B_s$ are given by
\begin{align*}
    A_s &:=e_{s+1}(0)+(\frac{3}{2}-n)e_{s}(1)+(\frac{1}{2}n^2-\frac{3}{2}n+1)e_{s-1}(2)\in A^s({M_{d,\chi}}),\\
    B_s &:=e_{s+1}(0)+(\frac{1}{2}-n)e_{s}(1)+(\frac{1}{2}n^2-\frac{1}{2}n)e_{s-1}(2)\in A^s(M_{d,\chi}).
\end{align*}
 
\end{prop}

\begin{proof}
We expand $\mathrm{ch}(Rr_* \mathcal{H}(n))$ as above, which gives
\begin{align*}
\mathrm{ch}(Rr_* \mathcal{H}(n))&=r_*(\mathrm{ch}(\mathbb{F}^\vee)\cdot \mathrm{ch}(\mathcal{O}_Z\otimes \mathcal{O}_{\BP^2}(-n))\cdot \mathrm{td}(\BP^2))\\
&=r_*(\mathrm{ch}(\mathbb{F}^\vee)\cdot (1-\mathrm{exp}(-H - \beta))\cdot \exp (-nH)\cdot \mathrm{td}(\mathbb{P}^2))\\
&=r_*(\mathrm{ch}(\mathbb{F}^\vee)\cdot \exp (-nH)\cdot \mathrm{td}(\mathbb{P}^2))-r_*(\mathrm{ch}(\mathbb{F}^\vee)\cdot \mathrm{exp}(-(n+1)H - \beta)\cdot \mathrm{td}(\mathbb{P}^2)).\tag{$\dagger$}
\end{align*}
Here the second equality follows from Lemma \ref{lem2.4}, and we temporarily suppress the various pull-back maps in the expression for notational convenience. 

\smallskip

We set
\[
    \CL:=\sum_{k\geq 0} e_k(2),\quad
    \CM:=\sum_{k\geq 0} e_k(1),\quad
    \CN:=\sum_{k\geq 1} e_k(0),
\]
which are all finite sums. From the definition of the classes $e_k(j)$, we have
\[
\mathrm{ch}(\mathbb{F}^\vee)=\CL + \CM \cdot H +\CN \cdot H^2.
\]
Recall that $\mathrm{td}(\mathbb{P}^2)=1+\frac{3}{2}H +H^2$, the first term in $(\dagger)$ reads
\begin{align*}
r_*(\mathrm{ch}(\mathbb{F}^\vee)\cdot \exp (-nH)\cdot \mathrm{td}(\mathbb{P}^2))&=r_*((\CL+\CM\cdot H+ \CN\cdot H^2)\cdot \mathrm{exp}(-nH) \cdot \mathrm{td}(\mathbb{P}^2))\\
&=\CN+\big(\frac{3}{2}-n\big)\CM+\big(\frac{1}{2}n^2-\frac{3}{2}n+1\big)\CL\\
&=\sum_{s\geq 0} A_s.
\end{align*}
Similarly, the second term reads
\begin{align*}
r_*(\mathrm{ch}(\mathbb{F}^\vee)\cdot \mathrm{exp}(-(n+1)H - \beta)\cdot \mathrm{td}(\mathbb{P}^2))&=\mathrm{exp}(-\beta)\cdot\Big(\CN+\big(\frac{1}{2}-n\big)\CM+\frac{1}{2}\big(n^2-n\big)\CL\Big)\\
&=\exp(-\beta)\cdot \sum_{s\geq 0} B_s.
\end{align*}
Putting these together, we arrive at
\[
\mathrm{ch}_s (Rr_* \mathcal{H}(n))=\pi_M^* A_s-\sum_{\substack{i+j=s\\ 0\leq i\leq 2}} \frac{\pi_R^* \beta^i}{i!}(-1)^i \pi_M^* B_j.
\]
Let $x_1, x_2, \ldots, x_d$ be the Chern roots of the vector bundle $K^1$ in Lemma \ref{K^1}. Then we have
\[
c_\ell(-Rr_*\mathcal{H}(n))= c_\ell(K^1)=\mathbf{e}_\ell(x_1,\ldots, x_d),
\]
where $\mathbf{e}_\ell(x_1,\ldots, x_d)$ denotes the $\ell$-th elementary symmetric polynomials in the $x_i$. We write 
\[\mathbf{p}_\ell(x_1,\ldots, x_d):=\sum_{i=1}^d x_i^\ell
\]
the usual $\ell$-th power sums. Newton's identity then gives
\[
\mathbf{e}_\ell=(-1)^\ell \sum_\mathbf{m}\prod_{s=1}^\ell \frac{(-\mathbf{p}_s)^{m_s}}{m_s! s^{m_s}},
\]
with the summation taken over $\mathbf{m}$ as in the proposition. The desired vanishing now follows from Corollary \ref{cor2.3} and the identity $\mathbf{p}_s=-s!\cdot  \mathrm{ch}_s(Rr_*\mathcal{H}(n))$.
\end{proof}

Before going into details in Section \ref{Sec2.3}, we briefly explain our strategy to produce relations. 

\smallskip

The expression in (\ref{relation}) is a polynomial in the variable $\pi_R^* \beta$ of degree $\leq 2$ . We can integrate it with respect to $\pi_R^* \beta^j$ over $M_{d,\chi} \times \check{\BP}^2$:
\[
\pi_{M*}(\pi_R^* \beta^j \cdot (\textup{Expansion of LHS of (\ref{relation})})) = 0.
\]
This amounts to picking out the coefficient of the term $\pi_R^* \beta^{2-j}$, for $j\in \{0,1,2\}$. The projection formula then gives relations among $e_k(j)$ in $A^*(M_{d,\chi})$. Moreover, for a fixed $j$ we let $n$ take $\{1,2,3\}$; this provides \textit{three} expressions of similar types. Taking linear combinations of these, we obtain desired relations that kill certain leading terms.

\subsection{Explicit computations}\label{Sec2.3} Although the expression (\ref{relation}) looks complicated, it is not hard to read the coefficients of leading terms in a given degree. We illustrate in the following that we can kill all $e_k(j)$ in degrees $\geq d-1$ via the relations produced by (\ref{relation}).

\subsubsection*{Step 1: Kill $A^{\geq d+1}(M_{d,\chi})$ and $e_{d+1}(0) \in A^d(M_{d,\chi})$.}

For each $\ell \geq d+1$, we first integrate (\ref{relation}) with respect to $\pi_R^* \beta^2$ for $n=1,2,3$, and consider the coefficients of the three leading terms (with respect to the order $\prec$) $e_{\ell+1}(0), e_\ell(1), e_{\ell-1}(2)$. 

Since these classes only appear in $A_\ell$ and $B_\ell$, the single relevant $\ell$-tuple $\mathbf{m}$ in the sum is just $\mathbf{m}=(0,0,\ldots, 1)$. We look at the corresponding summand in (\ref{relation}), and the coefficients of the three classes come from integrating the term
\[
(\ell-1)!\cdot (\pi_M^* A_\ell -\pi_M^* B_\ell).
\]
Denoting by $R_i$ the expression obtained by setting $n=i$, we obtain via a direct calculation the following table of coefficients for the leading terms, where we scale all coefficients below by $\frac{1}{(\ell-2)!}$ for convenience:
\begin{center}
\medskip
\begin{tabular}{||c c c c||}
 \hline
  & $R_1$ & $R_2$ & $R_3$ \\ 
 \hline\hline
 $e_{\ell+1}(0)$ & $0$ & $0$ & $0$ \\ 
 \hline
 $e_{\ell}(1)$ & $\ell-1$ & $\ell-1$ & $\ell-1$ \\
 \hline
 $e_{\ell-1}(2)$ & $0$ & $-(\ell-1)$ & $-2(\ell-1)$ \\
 \hline
\end{tabular}
 \medskip
\end{center}
In particular $R_1$ kills the class $e_\ell(1)$, while $R_1-R_2$ kills the class $e_{\ell-1}(2)$. As this holds for $\ell\geq d+1$, we kill all classes $e_k(j)\in A^{\geq d+1}(M_{d,\chi})$ with $j\in \{1,2\}$. 

We still need to treat $e_k(0) \in A^{\geq d+1}(M_{d,\chi})$. For this purpose, we integrate (\ref{relation}) with respect to $\pi_R^* \beta$ for $n=1,2,3$. This produces relations in $A^{\ell-1}(M_{d,\chi})$. Similarly, we look at the coefficients of the three leading terms $e_{\ell}(0), e_{\ell-1}(1), e_{\ell-2}(2)$. These classes appear in $A_{\ell-1}$ and $B_{\ell-1}$, so there are two relevant $\mathbf{m}$, namely $(0,0,\ldots,0,  1)$ and $(1,0,\ldots, 1,0)$. The coefficients of the three terms come from integrating
\begin{align*}
    (\ell-1)!\cdot \pi_R^* \beta \cdot \pi_M^* B_{\ell-1}+
    (\ell-2)!\cdot \pi_R^* \beta \cdot \pi_M^* B_0\cdot (\pi_M^* A_{\ell-1}-\pi_M^* B_{\ell-1}).
\end{align*}

We see that the coefficient of $e_{\ell}(0)$ in $R_1$ is $\ell-1$, so this kills the classes $e_k(0) \in A^{\geq d}(M_{d,\chi})$ by definition of the order $\prec$.

\subsubsection*{Step 2: Kill $~e_{d-1}(2) \in A^d(M_{d,\chi})$}
We have already killed all classes $e_k(j)$ lying in $A^{\geq d+1}(M_{d,\chi})$. In addition, we have also killed $e_{d+1}(0)$ as explained at the end of Step 1. We now kill $e_{d-1}(2)$.

We integrate the relation (\ref{relation}) for $\ell=d+1$ over the class $\pi_R^*\beta$. From the same calculation as above, the coefficients of $e_{d+1}(0), e_d(1), e_{d-1}(2)$ come from the terms
\begin{align*}
    d!\cdot \pi_R^* \beta \cdot \pi_M^* B_{d}+
    (d-1)!\cdot \pi_R^* \beta \cdot \pi_M^* B_0\cdot (\pi_M^* A_{d}-\pi_M^* B_{d}).
\end{align*}
Using Lemma \ref{lem2.1}, we easily compute that 
\[
\pi_M^* B_0=\chi+(n-2)d \in A^0(M_{d,\chi} \times \check{\mathbb{P}}^2).
\]
This gives the following table of coefficients:

\begin{center}
\medskip
\begin{tabular}{||c c c c||}
 \hline
  & $R_1$ & $R_2$ & $R_3$ \\ 
 \hline\hline
 $e_{d+1}(0)$ & $d$ & $d$ & $d$ \\ 
 \hline
 $e_{d}(1)$ & $\chi-\frac{3}{2}d$ & $\chi-\frac{3}{2}d$ & $\chi-\frac{3}{2}d$ \\
 \hline
 $e_{d-1}(2)$ & $0$ & $d-\chi$ & $d-2\chi$ \\
 \hline
\end{tabular}
 \medskip
\end{center}

Since $d\neq \chi$, the linear combination $R_1-R_2$ kills $e_{d-1}(2)$; on the other hand, we see from the table that the classes $e_{d}(1)$ cannot be killed using the relations $R_i$ above as the first two rows are proportional.

\subsubsection*{Step 3: Kill~ $A^{d-1}(M_{d,\chi})$} Next, we integrate (\ref{relation}) with respect to $\pi_R^*(\mathbf{1}_{\mathbb{P}^2})$, leading to relations in $A^{\ell-2}(M_{d,\chi})$. There are four relevant $\mathbf{m}$, namely 
 \[
 \mathbf{m}\in\{(0,0,\ldots, 0,0,1), (1,0,\ldots, 0,1,0), (0,1,\ldots, 1,0,0), (2,0,\ldots, 1,0,0)\}.
 \] 
 Looking at $\ell=d+1$, a direct computation gives the following table of coefficients. The computations are similar as above but involve more terms, so we leave the details to readers. Note that we used $d\geq 4$ here to ensure that $(0,1,\ldots, 1,0,0)\neq (0,2,0,0)$. 
 
 \begin{center}
\medskip
\begin{tabular}{||c c c c||}
 \hline
  & $R_1$ & $R_2$ & $R_3$ \\ 
 \hline\hline
 $e_{d}(0)$ & $\chi-\frac{3}{2}d$ & $\chi-\frac{1}{2}d$ & $\chi+\frac{1}{2}d$ \\ 
 \hline
 $e_{d-1}(1)$ & $\frac{-2\chi^2+6\chi d-5d^2 +d}{4(1-d)}$ & $\frac{-2\chi^2+6\chi d -4 \chi -3d^2 +3d}{4(1-d)}$ & $\frac{-2\chi^2 +6 \chi d -8\chi +3d^2 -3d}{4(1-d)}$ \\
 \hline
 $e_{d-2}(2)$ & $0$ & $\frac{\chi^2-2\chi d +\chi +d^2 -d}{2(1-d)}$ & $\frac{2\chi^2-2\chi d +4 \chi -d^2+d}{2(1-d)}$ \\
 \hline
\end{tabular}
 \medskip
\end{center}
 
 Denote this $3\times 3$ matrix by $\mathsf{M}$. To kill all three classes $e_d(0), e_{d-1}(1), e_{d-2}(2)$, it suffices to show that $\mathsf{M}$ is nonsingular, \textit{i.e.,} $\det(\mathsf{M})\neq 0$. In fact, a direct calculation yields
 \[
 \det(\mathsf{M})=\frac{1}{4(1-d)^2} \chi(d-2)(d-\chi)(d-2\chi) \neq 0.
 \]

 \subsubsection*{Step 4: Kill the remaining class $e_d(1)$}
 We look at $\ell=d+2$ and still integrate (\ref{relation}) with respect to $\pi_R^*(\mathbf{1}_{\mathbb{P}^2})$. A similar calculation gives
 
\begin{center}
\medskip
\begin{tabular}{||c c c c||}
 \hline
  & $R_1$ & $R_2$ & $R_3$ \\ 
 \hline\hline
 $e_{d+1}(0)$ & $\chi-\frac{3}{2}d-\frac{1}{2}$ & $\chi-\frac{d}{2}-\frac{1}{2}$ & $\chi+\frac{d}{2}-\frac{1}{2}$ \\ 
 \hline
 $e_{d}(1)$ & $\frac{5d^2-6\chi d+3d +2\chi^2 -2\chi}{4d}$ & $\frac{3d^2 -6\chi d +3d +2\chi^2 -2\chi}{4d}$ & $\frac{-3d^2-6\chi d +3d+2\chi^2 -2\chi}{4d}$ \\
 \hline
\end{tabular}
 \medskip
\end{center}
 
Consider the linear combination
 \[
 S:=(\chi-\frac{d}{2}-\frac{1}{2})R_1-(\chi-\frac{3}{2}d-\frac{1}{2})R_2,
 \]
 and a computation shows that the coefficient of $e_d(1)$ in $S$ is $\frac{1}{2}(d-\chi)(d-\chi+1)$. Since we have assumed that $0<\chi<d$, this coefficient will never vanish. Thus we can kill $e_d(1)$ as well.
 
 \medskip
  In conclusion, we have produced relations in $A^*(M_{d,\chi})$ with leading terms $e_k(j)$, for every tautological class $e_k(j)\in A^{\geq d-1}(M_{d,\chi})$. Combined with Proposition \ref{prop2.2}, this proves Theorem \ref{thm1} (a). Henceforth, we will call the $3d-7$ classes in (\ref{taut}) the \textit{tautological generators.} \qed
    
\section{Freeness and further discussions}

\subsection{Freeness.} In this section we prove Theorem \ref{thm1} (b), \emph{i.e.,} there is no relations among $c_k(j)$ in $A^{\leq d-2}(M_{d,\chi})$. Recall that the cycle class map \[\mathrm{cl}: \mathrm{CH}^*(M_{d,\chi}) \to H^{2*}(M_{d,\chi},\mathbb{Q})
\]
is isomorphic \cite{Integral}. Therefore to prove freeness results, it suffices to work with the cohomology ring.

We first recall the following theorem recently obtained by Yuan \cite{YY4} concerning the Betti numbers, which relies on a dimension estimate of \cite{MS_GT}.

\begin{thm}[\cite{YY4}, Theorem 1.5]
    \label{lem3.1}
   For even integers $0\leq i\leq 2d-4$, we have an equality between Betti numbers
    \[
    b_i(M_{d,\chi})=b_i({\BP^2}^{[\frac{d(d-3)}{2}-\chi_0]}),
    \]
    where ${\BP^2}^{[n]}$ denotes the Hilbert scheme of $n$-points on $\BP^2$, and $\chi_0 \equiv \chi \mod d$ with $-2d-1\leq \chi_0\leq -d+1$.
\end{thm}

This theorem allows us to compute the Betti numbers for $M_{d,\chi}$ by the (known) Betti numbers of Hilbert scheme of points. Note that by \cite{MS_GT}, to apply Theorem \ref{lem3.1} we can take $\chi=1$, so that $\chi_0=-d+1$ and $\frac{d(d-3)}{2}-\chi_0=\frac{(d+1)(d-2)}{2}$. In particular, for fixed $d\geq 4$ we have 
\begin{equation}\label{betti1}
b_{2d-4}(M_{d, \chi})=b_{2d-4}({\BP^2}^{[\frac{(d+1)(d-2)}{2}]}).\end{equation}

Our goal is to show that the invariant (\ref{betti1}) matches exactly the number of monomials formed by $c_{k}(j)$ of cohomological degree $2d-4$, where there are two generators of degree $2$ and three generators for each degree $4,6, \ldots, 2d-4$. Combined with Proposition \ref{thm1.5}, this will imply the freeness in $H^{\leq 2d-4}(M_{d,\chi}, \mathbb{Q})$.

To this end, we invoke G\"ottsche's formula on Betti numbers of Hilbert schemes of points. Denote by $P(-,z):=\sum_i (-1)^i b_i(-) z^i$ the Poincaré polynomial.

\begin{thm}[\cite{Go1}]
\label{gottsche}
Let $S$ be a nonsingular projective surface. We have an identity between two series in variables $z$ and $t$:
\[
\sum_{n\geq 0} P(S^{[n]},z)t^n =\prod_{k\geq 1} \prod_{i=0}^4 (1-z^{2k-2+i}t^k)^{(-1)^{i+1}b_i(S)}.
\]
\end{thm}

Specializing to $S=\mathbb{P}^2$, Theorem \ref{gottsche} yields
\[
\sum_{n\geq 0} \Big(\sum_{i} (-1)^i b_i(X^{[n]})z^i\Big) t^n =\prod_{k\geq 1} (1-z^{2k-2} t^k)^{-1}(1-z^{2k}t^k)^{-1}(1-z^{2k+2}t^k)^{-1}.
\]

In order to look at the coefficient of $z^{2d-4}t^{\frac{(d+1)(d-2)}{2}}$ in the right-hand side, we expand it into a power series:
\begin{align*}
\mathrm{RHS} = \; &(1+t+t^2+\cdots)(1+z^2 t+ z^4 t^2+\cdots)(1+z^4 t+z^8 t^2+\cdots)\\&(1+z^2 t^2+z^4t^4+\cdots)(1+z^4t^2+z^8t^4+\cdots)(1+z^6t^2+z^{12}t^4+\cdots)\cdots.
\end{align*}
Note that $2d-4 < \frac{(d+1)(d-2)}{2}$ whenever $d\geq 4$. Taking into account the first term $(1+t+t^2+\cdots)$, we see that it suffices to count the number of monomials with degree $2d-4$ in the expansion 
\[
(1+z^2+ z^4 +\cdots)(1+z^4 +z^8 +\cdots)(1+z^2 +z^4+\cdots)(1+z^4+z^8+\cdots)(1+z^6+z^{12}+\cdots)\cdots.
\]
The result then follows from the combinatorial operation to `replace' $z$ by the tautological generators $c_k(j)$. More precisely, we compare the above expansion with the generating series
\[
\big[\sum_{i\geq 0 }c_0(2)^i\big]\big[\sum_{i\geq 0 }c_1(2)^i\big]\big[\sum_{i\geq 0 }c_2(0)^i\big] \big[\sum_{i\geq 0 }c_2(1)^i\big] \big[\sum_{i\geq 0 }c_2(2)^i\big]\cdots
\]
where we have two generators in cohomological degree 2 and three in each degree among $4, 6, \ldots, 2d-4$. We conclude that $b_{2d-4}(M_{d,\chi})$ matches with the \textit{maximal possible} dimension, which implies that no relations can occur among the tautological generators. This completes the proof of Theorem \ref{thm1} (b). \qed

\subsection{Numerical data and further discussions.}
\label{sec3.2}
Numerically Betti numbers for low values of $d$ can be obtained using various methods \cite{CKK, CC15, CC16, PB}. The following tables summarize the numerical data for $4\leq d \leq 9$ and $\chi=1$. The second rows list the actual Betti numbers, and the third rows compute the dimensions \textit{assuming freeness} (\emph{i.e.,} there is no relation) of the tautological generators.

\begin{center}
\medskip
\begin{tabular}{||c c c c c c||} 
 \hline
  $M_{4,1}$ & $H^0$ & $H^2$ & $H^4$ & $H^6$ & $H^8$\\ 
 \hline
 actual dim & 1 & 2 & 6 & 10 & 14  \\
 \hline
 freeness & 1 & 2 & 6 & 10 & 20\\
 \hline
\end{tabular}
\medskip
\end{center}

\begin{center}
\begin{tabular}{||c c c c c c c||} 
 \hline
  $M_{5,1}$ & $H^0$ & $H^2$ & $H^4$ & $H^6$ & $H^8$ & $H^{10}$\\ 
 \hline
 actual dim & 1 & 2 & 6 & 13 & 26 & 45 \\
 \hline
 freeness & 1 & 2 & 6 & 13 & 26 & 48\\
 \hline
\end{tabular}
\medskip
\end{center}

\begin{center}
\begin{tabular}{||c c c c c c c c||} 
 \hline
  $M_{6,1}$ & $H^0$ & $H^2$ & $H^4$ & $H^6$ & $H^8$ & $H^{10}$ & $H^{12}$\\ 
 \hline
 actual dim & 1 & 2 & 6 & 13 & 29 & 54 & 101 \\
 \hline
 freeness & 1 & 2 & 6 & 13 & 29 & 54 & 104\\
 \hline
\end{tabular}
\medskip
\end{center}

\begin{center}
\begin{tabular}{||c c c c c c c c c||} 
 \hline
  $M_{7,1}$ & $H^0$ & $H^2$ & $H^4$ & $H^6$ & $H^8$ & $H^{10}$ & $H^{12}$ & $H^{14}$\\ 
 \hline
 actual dim & 1 & 2 & 6 & 13 & 29 & 57 & 110 & 196 \\
 \hline
 freeness & 1 & 2 & 6 & 13 & 29 & 57 & 110 & 199\\
 \hline
\end{tabular}
\medskip
\end{center}

\begin{center}
\begin{tabular}{||c c c c c c c c c c||} 
 \hline
  $M_{8,1}$ & $H^0$ & $H^2$ & $H^4$ & $H^6$ & $H^8$ & $H^{10}$ & $H^{12}$ & $H^{14}$ & $H^{16}$\\ 
 \hline
 actual dim & 1 & 2 & 6 & 13 & 29 & 57 & 113 & 205 & 369 \\
 \hline
 freeness & 1 & 2 & 6 & 13 & 29 & 57 & 113 & 205 & 372\\
 \hline
\end{tabular}
\medskip
\end{center}

\begin{center}
\begin{tabular}{||c c c c c c c c c c c||} 
 \hline
  $M_{9,1}$ & $H^0$ & $H^2$ & $H^4$ & $H^6$ & $H^8$ & $H^{10}$ & $H^{12}$ & $H^{14}$ & $H^{16}$ & $H^{18}$\\ 
 \hline
 actual dim & 1 & 2 & 6 & 13 & 29 & 57 & 113 & 208 & 378 & 657 \\
 \hline
 freeness & 1 & 2 & 6 & 13 & 29 & 57 & 113 & 208 & 378 & 660\\
 \hline
\end{tabular}
\medskip
\end{center}

\smallskip

From these data, it is reasonable to form the following conjecture.

\begin{conj}\label{conj3.3}
    The relations of the least degree among the tautological generators (\ref{taut}) in $A^*(M_{d,\chi})$ occur in $A^d(M_{d,\chi})$. In particular, there is no relation  in degree $d-1$ as well.
\end{conj}

The freeness in degree $d-1$ is equivalent to 
\[
b_{2d-2}(M_{d,\chi}) =b_{2d-2}({\BP^2}^{[\frac{(d+1)(d-2)}{2}]}) -3.
\]
In particular, if one can calculate this Betti number using the algorithm of \cite{PB} or the motivic method of \cite{YYY}, then the freeness part of the conjecture follows. In general, a more systematic study of tautological relations is needed to fully understand the relations among (\ref{taut}).

\begin{rmk}\label{llast}
Conjecture \ref{conj3.3} has already been proven recently by Yuan \cite{YY5}.
\end{rmk}

\end{document}